\documentclass[a4paper,12pt,reqno]{amsart}
\usepackage{amsmath, amssymb, amsthm, epsfig}
\usepackage{hyperref, latexsym}
\usepackage{url}
\usepackage[mathscr]{euscript}
\usepackage{color}
\usepackage{harpoon}
\usepackage{url}
\usepackage{pgfplots}
\usepackage{tikz}
\usepackage{adjustbox}
\usepackage{subcaption}

	\usetikzlibrary{automata,topaths}
	\usetikzlibrary{plotmarks}
	\usetikzlibrary{patterns}
	\usetikzlibrary{svg.path} 
\usepackage[height=23.6 cm, width=16.4cm, centering]{geometry}

\usepackage{refcheck} 
\norefnames
\nocitenames
\usepackage{mathtools}
\mathtoolsset{showonlyrefs}

\def\today{\ifcase\month\or
  January\or February\or March\or April\or May\or June\or
  July\or August\or September\or October\or November\or December\fi
  \space\number\day, \number\year}

\newtheorem{theorem}{Theorem}
\newtheorem{lemma}{Lemma}
\newtheorem{proposition}{Proposition}
\newtheorem{corollary}{Corollary}

\newcommand{\C}{\mathcal{C}}

\renewcommand{\H}{\mathcal{H}}

\renewcommand{\L}{\mathcal{L}}

\renewcommand{\P}{\mathcal{P}}
\newcommand{\Q}{\mathcal{Q}}
\newcommand{\R}{\mathcal{R}}

\newcommand{\Z}{\mathcal{Z}}

\newcommand{\QQ}{\mathfrak{Q}}

\renewcommand{\r}{\mathbb{R}}
\newcommand{\cp}{\mathbb{C}} 

\newcommand{\re}{{\rm Re}\,}
\newcommand{\ft}{\widehat}

\newcommand{\bo}{\boldsymbol}

\newcommand{\bn}{{\bo n}}
\newcommand{\bm}{{\bo m}}

\newcommand{\om}{\omega}
\newcommand{\la}{\lambda}

\newcommand{\al}{\alpha}
\newcommand{\ep}{\varepsilon}

\newcommand{\dist}{{\rm dist}}

\renewcommand{\phi}{\varphi}
\newcommand{\1}{\mathbf{1}}

\renewcommand{\d}{\,{\rm d}}  

\def\be{\begin{equation}}  
\def\ee{\end{equation}}

\def\R{\mathbb R}  \def\Q{\mathbb Q} \def\C{\mathbb C} \def\Z{\mathbb Z} 
\def\a{\alpha} \def\b{\beta} \def\l{\lambda} \def\L{\Lambda} \def\G{\Gamma} \def\QQ{\mathcal Q}
\def\={\;=\;} \def\+{\;+\;} \def\wt{\widetilde}
\renewcommand{\k}{\kappa}
\def\sp{\vphantom{\bigl|^\lambda}}

\def\R{\mathbb R}  \def\Q{\mathbb Q} \def\C{\mathbb C} \def\Z{\mathbb Z} 
\def\a{\alpha} \def\b{\beta} \def\l{\lambda} \def\L{\Lambda}
\def\={\;=\;} \def\+{\;+\;} \def\wt{\widetilde}
\renewcommand{\k}{\kappa}
\def\sp{\vphantom{\bigl|^\lambda}}
 
\RequirePackage[normalem]{ulem} 
\RequirePackage{color}\definecolor{RED}{rgb}{1,0,0}\definecolor{BLUE}{rgb}{0,0,1} 
\RequirePackage{listings} 
\RequirePackage{color} 
\lstdefinelanguage{DIFcode}{ 
  moredelim=[il][\color{red}\sout]{\%DIF\ <\ }, 
  moredelim=[il][\color{blue}\uwave]{\%DIF\ >\ } 
} 
\lstdefinestyle{DIFverbatimstyle}{ 
	language=DIFcode, 
	basicstyle=\ttfamily, 
	columns=fullflexible, 
	keepspaces=true 
} 
\lstnewenvironment{DIFverbatim}{\lstset{style=DIFverbatimstyle}}{} 
\lstnewenvironment{DIFverbatim*}{\lstset{style=DIFverbatimstyle,showspaces=true}}{} 

\begin{document}
\setlength{\parskip}{2pt}
\setlength{\parindent}{16pt}

\title[]{Strichartz estimates with broken symmetries}
\author[Gon\c{c}alves and Zagier]{Felipe Gon\c{c}alves and Don Zagier}
\subjclass[2010]{}
\address{Hausdorff Center for Mathematics,
Endenicher Allee 60, 53115 Bonn, Germany}
\email{goncalve@math.uni-bonn.de}
\address{Max Planck Institute for Mathematics, Vivatsgasse 7, 53111 Bonn, Germany \newline
 \phantom{Int}International Centre for Theoretical Physics, Strada Costiera 11, 34151 Trieste, Italy}
\email{dbz@mpim-bonn.mpg.de}
\allowdisplaybreaks
\maketitle

\begin{abstract}
In this note we study the eigenvalue problem for a quadratic form associated with Strichartz 
estimates for the Schr\"odinger equation, proving in particular a sharp Strichartz inequality
for the case of odd initial data.  We also describe an alternative method that is applicable to 
a wider class of matrix problems.
\end{abstract}

\section{Introduction}
In this paper we study the eigenvalue problem associated with the quadratic form
\footnote{The $\sqrt{12}$ normalization factor has been included here for aesthetic reasons only.}
\begin{align}\label{def:Q}
Q(f) \ =\sqrt{12}\int_{\r} \int_{\r} |u_f((y,y,y),t)|^2 \d y \d t 
\end{align}
restricted to the subspace of functions $f\in L^2(\r^3)$ with some prescribed parity.  Here 
$u_f: \r^3\times \r\to \cp$ denotes the solution of the Schr\"odinger equation
\be \label{Schroed} \partial_t u(x,t) \,=\, i\Delta u(x,t)\,, \qquad  u(x,0)\,=\,f(x)   \ee
with initial data~$f$ and we say that $f(x_1,x_2,x_3)$ has parity $\ep\in\{\pm 1\}^3$ if it is even with respect
to the $x_i$ with $\ep_i=1$ and odd with respect to the $x_i$ with $\ep_i=-1$. {Enforcing such parity constraints on the 
initial data breaks the fundamental symmetries associated with the Schr\"odinger equation (for instance Galilean invariance).}
The study of such quadratic form is motivated by its intrinsic relation with Strichartz estimates, 
as it was used in~\cite{Ca,HZ} to produce sharp bounds for the space-time $L^6$-norm of $u_f$ in one space dimension.
For a more general overview of how results of this sort are used we refer to the recent survey \cite{FO}.

Since the form~$Q$ is invariant under permutation of the coordinates, we can restrict to the case where
$\ep_i$ is $-1$ for $i\le\k$ and $+1$ for $i>\k$ for some $\k\in\{0,1,2,3\}$, i.e., to
the subspace $L_\k^2(\R^3)\subset L^2(\r^3)$ of functions that are odd with respect to the first~$\k$
variables and even with respect to the others.
It turns out, as we will show, that the eigenvalues associated to the restriction of~$Q$ to each $L_\k^2(\R^3)$ are 
all {\it rational} and can be given explicitly as the coefficients of certain algebraic generating functions 
(which was amusing and unexpected for the authors). Specifically, define four power series 
$G_\k(w)\in\Q[[w]]$ by 
\be\label{G0G1G2G3}
\begin{pmatrix} G_0(w) \\ G_1(w)\sp \\ G_2(w)\sp\\ G_3(w)\sp \end{pmatrix}
 \=  \begin{pmatrix} 1&3&3 \\ w&1&-1\sp \\ 1&-1&-1\sp \\ w&-3&3\sp \end{pmatrix}\,
  \begin{pmatrix}\frac1{4(1-w^2)}\\ \frac1{8\sqrt{1-\frac23w+w^2}}\sp  \\ \frac1{8\sqrt{1+\frac23w+w^2}}\sp \end{pmatrix}
\= \begin{pmatrix}1 + \frac7{27}w^4 + \frac{32}{81}w^6+\cdots \\ \frac13w+\frac4{27}w^3+\frac13w^5+\cdots\sp \\
\frac13w^2+\frac{20}{81}w^4+\frac{49}{243}w^6+\cdots\sp \\ \frac59w^3+\frac{91}{243}w^7+\frac{1760}{6561}w^9+\cdots\sp\end{pmatrix}
\ee
and denote by $\L_\k\subset\Z\bigl[\frac13\bigr]$ the multiset of coefficients of~$G_\k(w)$, counted with multiplicity.  
Since\footnote{Here we use the standard notation $\bigl[w^n\bigr](F)=c_n$ if $F(w)=\sum_{n=0}^\infty c_n w^n$.} $\bigl[w^n\bigr](G_\k)=\frac14 +O(n^{-1/2})$ for $n\equiv\k\pmod2$ (see Figure \ref{plot}) and
vanishes otherwise, the only non-zero limit point of~$\L_\k$ is $\frac14$ and this is also the only value
that can have infinite multiplicity\footnote{This could happen only
if $\frac14$ occurred infinitely often as a Taylor coefficient of~$G_\k$. In fact, numerical computations 
up to $n=10^5$ suggest that it never occurs at all, but we could not prove this.}. Our main result is the following.

\begin{figure*}
\centering
\begin{tikzpicture}[scale=1]
\begin{axis}[
xlabel = $n$,
ytick={0,.25,0.5,0.75,1.0},
]
\addplot[
    domain=1:500,
	only marks,
	mark size=0.5pt,
]
table {G0.data};
\end{axis}
\end{tikzpicture}
~
\begin{tikzpicture}[scale=1]
\begin{axis}[
xlabel = $n$,
ymax=.35, ymin=0.15,
ytick={0.2,.25,0.3},
]
\addplot[
   domain=1:500,
	only marks,
	mark size=0.5pt,
]
table {G1.data};
\end{axis}
\end{tikzpicture}
\caption{First $501$ nonzero power series coefficients of $G_{0}$ (left) and $G_{1}$ (right).}
\label{plot}
\end{figure*}
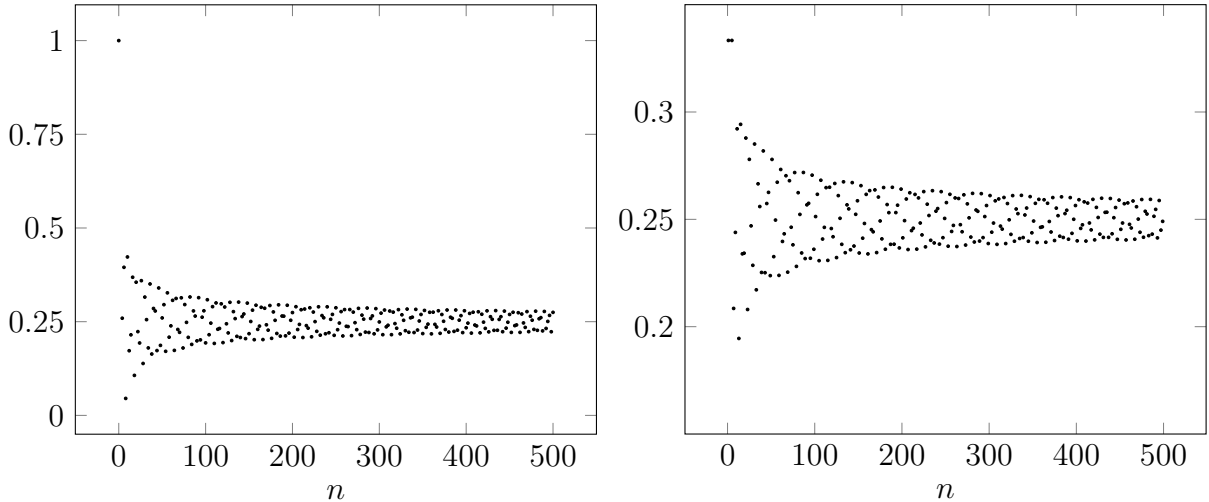

\begin{theorem}\label{thm:paritymain}
For each $\k=0,1,2,3$ there is an orthogonal decomposition 
$$ L^2_\k(\r^3)\=\bigoplus_{\la \in \Lambda_\k}  V_{\k}(\la)\,, $$
that diagonalizes $Q$ such that $Q(f)=\l\|f\|^2_{L^2(\r^3)}$ for $f\in V_\k(\l)$.  In particular, the largest eigenvalue 
$\al_\k=\max(\Lambda_\k)$ of~$Q$ on $L^2_\k$ is given explicitly by~$\a_0=1$, $\a_1=\a_2=\frac13$, and $\a_3=\frac59$. 
The eigenspaces $V_\k(\a_\k)$ corresponding to these largest eigenvalues are given explicitly by
\begin{align*}
V_0(1) & \=  \left\{f\in L^2(\r^3): f(x)=g(|x|)\right\}, \\
V_1(\tfrac13) & \= \left\{f\in L^2(\r^3): f(x)=x_1g(|x|)+x_1P(x_1,x_2,x_3)h(|x|)\right\}, \\
V_2(\tfrac13) & \= \left\{f\in L^2(\r^3): f(x)=x_1x_2 g(|x|)\right\}, \\
V_3(\tfrac59) & \=\left\{f\in L^2(\r^3): f(x)=x_1x_2x_3g(|x|)\right\},
\end{align*}
where $P(x_1,x_2,x_3)\,=\,x_1^4 - 5(x_2^2 + x_3^2)x_1^2 +15 (x_2^4 - 5x_3^2x_2^2 + x_3^4)$.
\end{theorem}
The proof will actually identify the set of eigenvalues of~$Q$ on $L^2_\k$ with~$\L_\k$ as
a multiset, in the sense that each $V_\k(\l)$ is a module over the set $V_0(1)$ of radial $L^2$-functions 
of rank equal to the multiplicity of~$\l$ in~$\L_\k$ (and hence finite for $\l$ different from~0 and
possibly~$\frac14$). The fact that $V_1(\tfrac13)$ looks ``larger" than the other three $V_\k(\al_\k)$ is then
simply a reflection of the fact that $\frac13$ appears twice as a coefficient in the power series~$G_1(w)$.

Theorem 1 gives not only the largest eigenvalue, but the whole spectrum for each $\k$. In particular, one can check by a numerical computation (Lemma \ref{lem:boundslambda} below) that the second largest eigenvalue $\b_\k$ of $Q$ on $L^2_\k$ is given by
$$
(\b_0,\b_1,\b_2,\b_3)=\bigg(\frac{8320}{3^9},\frac{469136}{3^{13}},\frac{232}{3^6},\frac{221312}{3^{12}}\bigg).
$$
Note that $\a_\k>\beta_\k$ for $\k=0,1,2,3$. From Theorem~\ref{thm:paritymain} we establish a generalised version of a conjecture of the first author \cite[Conjecture 1]{Gon2} and, by restricting to tensor products, we also produce a sharp trilinear inequality.

\begin{corollary}\label{Cor2}
Let $\k=0,1,2,3$. We have
\[
\sqrt{12}\int_{\r} \int_{\r} |u_f((y,y,y),t)|^2 \d y \d t  \leq \alpha_\k \int_{\r^3}|f(x)|^2dx-(\a_\k-\beta_\k) \dist(f,V_\k(\al_k))^2
\]
for every $f \in L^2_{\k}(\r^3)$, with $\alpha_\k$, $\beta_\k$ and $V_\k(\alpha_\k)$ as above. Moreover, the constants $\alpha_\k$ and $\a_\k-\beta_\k$ above  are optimal. Furthermore,  if $f=g_1 \otimes g_2 \otimes g_3$ then
\begin{align*} 
   \sqrt{12}\int_{\r} \int_{\r} \;\bigl|u_{g_1}(y,t)u_{g_2}(y,t)u_{g_3}(y,t)\bigr|^2 \d y \d t  \leq\; \al_\k \prod_{j=1}^3 \int_{\r} |g_j(x)|^2 \d x
 \end{align*}
with equality if (and only if) $g_j(x)=b_jxe^{-ax^2}$ for $j\leq \k$ and $g_j(x)=b_je^{-ax^2}$ for $j>\k$ 
for some $a,\,b_j\in\C$ with $\re a>0$.
\end{corollary}

Here we have used the fact that the function $g_1\otimes g_2 \otimes g_3$ belongs to $V_\k(\al_k)$ 
if and only if  each $g_j(x)$ or $g_j(x)/x$, depending whether $g_j$ is even or odd, is a multiple of the same
Gaussian for $j=1,2,3$. (We leave the details to the reader). 

By setting $g_1=g_2=g_3$ in Corollary~\ref{Cor2} for an odd function $g_1$
we obtain the following sharp inequality for odd initial data.

\begin{corollary}\label{cor2}
Let $f\in L^2(\r)$ be an odd function and let $u_f:\r\times \r\to \cp$ solve the Schr\"odinger 
equation~\eqref{Schroed} with initial data $f$. Then
\begin{align}\label{eq:sharper66odd}
 \|u_f\|_{L^{6}(\r\times \r)} \;\,\leq\;\, 12^{-1/12}\,\bigl(5/9\bigr)^{1/6}\,\|f\|_{L^2(\r)}\;,
 \end{align}
with equality if and only if $f(x)=bxe^{-ax^2}$ with $a,b\in\C$ and $\,\re a>0\,$.
\end{corollary}

We add a few more remarks
to clarify the relation of the results here to earlier ones
in the literature. Ozawa and Tsutsumi proved inequality \eqref{eq:sharper66odd} with no parity condition on $g$ and 
optimal constant $12^{-\frac{1}{12}}$, which is attained by Gaussians. Foschi~\cite{Fo} provided an alternative proof that also characterized Gaussians as the only extremizers. Corollary \ref{cor2} is a sharpening of this inequality when the initial data is odd. 

One may wonder if our techniques can be adapted to the case $k=2$ in \eqref{eq:stric-est} below. After tensorization, this leads us to consider the quadratic form 
$$
Q_2(f)=4\int_{\r}\int_{\r^2} |u_f(x_1,x_2,x_1,x_2,t)|^2 dx_1dx_2dt
$$
for $f:\r^4\to \cp$ and we can ask about the spectral decomposition of such quadratic form under parity constraints on $f$ (this was first considered in \cite[Conjecture 2]{Gon2} in the even case). By applying the general form of \cite[Lemma 6]{Ca}, this quadratic form can be identified as the $L^2(\r^4)$-norm of the projection of $f$ onto the subspace of functions invariant under rotations preserving the directions $(1,0,1,0)$ and $(0,1,0,1)$. The complication arises since the the restriction of such projection to a parity space is no longer diagonal, but only \emph{block} diagonal and we could not recognize the eigenvalues of such blocks in explicit analytic form. We leave this question to future work. Nevertheless, the first author has investigated the quadratic form $Q_2$ in \cite{Gon2} under the constraint that $f(x_1,x_2,x_3,x_4)$ depends only on $x_1^2+x_2^2$ and $x_3^2+x_4^2$. Indeed, using the method of  Section \ref{sec:moregen} one can reprove \cite[Theorem 3]{Gon2} (with $\gamma=3/4$ as it was conjectured; we omit the details).

\subsection{Background}
For a given $f\in L^2(\r^k)$ we denote by $u_f(x,t)$ the solution of the Schr\"odinger equation~\eqref{Schroed} 
with initial data $f$. The Strichartz inequality (see~\cite[Theorem 2.3]{Tao}) states that there exists~$A>0$
such that
\begin{equation}\label{eq:stric-est}
\|u_f\|_{L^{p}(\r^k\times \R)} \;\leq\; A\,\|f\|_{L^2(\r^k)}\,,
\end{equation}
where $p=2+4/k$. It is conjectured that Gaussians are the only maximizers of this inequality, that is, 
letting $A$ be the best possible constant, the inequality is attained if and only if $f(x)=c\exp(-a|x|^2 +bx)$, 
where $c,a\in \cp$, $b\in \cp^k$ and $\re a>0$. In lower dimensions $k=1,2$ this conjecture was solved by 
Foschi~\cite{Fo} and alternative approaches were given in~\cite{BBCH,Ca,Gon,Gon2,HZ,OT}.  The results 
of~\cite{Ca} are of special interest to this paper. By embedding the problem into a higher dimensional 
version and using the delta calculus tool, Carneiro~\cite[Lemma 6]{Ca} proved the following useful 
representation\footnote{Generalizing the work of~\cite{HZ}.}.

\begin{proposition} \label{prop:ufPEidentity}
If $u_f:\r^{d}\times \r\to \cp$ solves Schr\"odinger's equation with initial data \hbox{$f\in L^2(\r^{d})$} then
\be\label{eq:projrepres}
  c_d\,\int_\r \int_\r \,\bigl|u_f((y,\dots,y),t)\bigr|^2\d y\,\d t 
   \;=\,\int_{\R^d} \;\bigl|P_{E}\bigl(w_d\ft f\,\bigr)(x)\bigr|^2 \d x\,,
\ee
where $E\subset L^2(\r^{d})$ is the subspace of functions invariant under rotations fixing the direction 
$\1=\frac{1}{\sqrt{d}}(1,\dots,1)$, $P_{E}$ is the orthogonal projection onto $E$, 
  $$ w_d(x)\=\left(|x|^2 - (x\cdot \1)^2\right)^{\frac{d-3}{4}}\,,$$
$\ft f$ is the Fourier transform of $f$ defined by
  $$ \ft f(\xi) \= (2\pi)^{-d/2}\int_{\r^d}f(x)e^{-i\xi\cdot x}\d x\,, $$
and
  $$ c_d\= \sqrt{d}\; 2^{d-2}\,\pi^{(d-3)/2}\,\Gamma((d - 1)/2)\,. $$
\end{proposition}

Notice when $d=3$ and $f(x)=g(x_1)g(x_2)g(x_3)$ for some function $g\in L^2(\r)$ we have $u_f((y,y,y),t)=u_g(y,t)^3$, 
$w=1$ and $c_d={\sqrt{12}}$, and thus we obtain
\be
\sqrt{12}\int_{\r}\int_{\r} |u_g(y,t)|^6\d y\d t  
 \;=\, \int_{\r^{3}} |P_{E}(\ft f\,)(x)|^2 \d x \;\leq \int_{\r^{3}} |\ft f (x)|^2 \d x 
 \,=\,\biggl(\int_{\r} |g(x)|^2 \d x\biggr)^3.
\ee
This is the  sharp Strichartz inequality \eqref{eq:stric-est} for $k=1$. Note that equality 
is attained if and only if $\ft f \in E$, which in turn is equivalent to $g$ being a Gaussian\footnote{This 
is the proof method of~\cite{HZ} in a nutshell, which was later generalized by Carneiro~\cite{Ca}.}. Also note that the left hand side of \eqref{eq:projrepres} for $d=3$ is the quadratic form $Q(f)$ defined in \eqref{def:Q}.

\section{Proofs}

In this section we collect the necessary facts to prove our main result. Since in most part our techniques
are general, we will work with a higher dimensional analogue of $Q$ and later on specialize to dimension ~3. 

Let $\H_n$ denote the space of spherical harmonics of degree $n$ on the $(d-1)$-dimensional unit sphere 
$\mathbb{S}^{d-1}$ (i.e., the restrictions to $\mathbb{S}^{d-1}$
of degree~$n$ homogeneous polynomials annihilated by~$\Delta$).  Let also
$$
\H_n \otimes {\rm Radial}\=\bigl\{f\in L^2(\r^d): f(x)=Y_n(x/|x|)\,g(|x|) \text{ and } Y_n\in \H_n\bigr\}\,. 
$$
It is well known that 
$$
L^2(\r^d)\=\bigoplus_{n\geq 0} \H_n \otimes {\rm Radial},
$$
where the sum is orthogonal, that is, any function $f\in L^2(\r^d)$ can be expanded in the form
\begin{equation}\label{eq:fharmonicexpansion}
f(x)\=\sum_{n\geq 0} Y_n(x/|x|)\,g_n(|x|) \qquad (Y_n\in \H_n)\,.
\end{equation}

\begin{lemma}\label{lem:projEformula}
We have
\begin{align}\label{eq:projformula}
P_Ef(x)\= \sum_{n\geq 0} \frac{Y_n(\1)}{C_n^{d/2-1}(1)}\,C_n^{d/2-1}\Bigl(\frac{x}{|x|}\cdot \1\Bigr)\,g_n(|x|),
\end{align}
if $f\in L^2(\r^d)$ is expanded as in \eqref{eq:fharmonicexpansion}, where $C_n^{d/2-1}(z)$ is the Gegenbauer polynomial defined by
\begin{align}\label{eq:gengegen}
\bigl(1-2zw+w^2\bigr)^{1-d/2}\=\sum_{n\geq 0} w^n C_n^{d/2-1}(z).
\end{align}
\end{lemma}

\begin{proof}
Firstly, it is not hard to realize that
$$
P_Ef(x)\=\int_G f(\rho x)\d G(\rho) 
$$
where $G$ is the subgroup of $SO(d)$ that fixes the vector $\1$ and $\d G$ its induced $G$-invariant Haar
 probability measure. Secondly, noting that the function
$$ \om\,\in\,\mathbb S^{d-1}\;\mapsto\int_G Y_n(\rho\om)\d G(\rho) $$
still is an spherical harmonic of degree $n$ that depends only on $\om \cdot \1$, and must therefore be a 
multiple of the zonal harmonic $C_n^{d/2-1}(\om\cdot \1)$. We conclude that
$$ 
 \int_G Y_n(\rho\,\om)\d G(\rho)\=\frac{Y_n(\1)}{C_n^{d/2-1}(1)}\,
  C_n^{d/2-1}\bigl(\om\cdot \1\bigr) \qquad\bigl(\,\om\in\mathbb S^{d-1}\,\bigr),
$$
which proves the lemma. 
\end{proof}

In what follows we let 
\be\label{L2k}
 L_\k^2(\R^d) \=\bigl\{f\in L^2(\r^d): f \text{ is odd in } x_1,x_2,\dots,x_\k \text{ and even in }x_{\k+1},\dots,x_{d}\bigr\}
\ee
for $\k=0,1,\dots,d$. We also let $P_\k$ denote the orthogonal projection onto $L_\k^2(\R^d)$ and set $T_\k=P_\k P_E$. 
Abusing notation, we use the same symbol $T_\k$ to denote the restriction of $T_\k$ to ${\H_n}$; this is 
justified since $T_\k=\wt T_\k \otimes {\rm Id}$ on $\H_n\otimes {\rm Radial}$, and we identify $\wt T_\k$ 
with $T_\k$. Using the fact that 
$$
P_\k (Y)(\om) = \frac{1}{2^d}\sum_{\ep\,\in\,\{\pm 1\}^d} \ep_1\cdots\ep_\k\,Y(\ep_1 \om_1, \dots,\ep_d\om_d),
$$
we can apply~\eqref{eq:projformula} to obtain the following lemma.

\begin{lemma}\label{prop:Teigenvectors}
We have 
$$
T_\k( Y_n)(\om) \= \frac{Y_n(\1)}{C_n^{d/2-1}(1)}\,P_{d,\k,n}(\om)
$$
for every $Y_n\in \H_n$, where, for all $n\ge0$, we set 
$$
P_{d,\k,n}(\om)\=\frac1{2^d}\sum_{\ep\,\in\,\{\pm 1\}^d} \ep_1\cdots\ep_\k\,C_n^{d/2-1}\Bigl(\frac{\om\cdot\ep}{\sqrt{d}}\Bigr)
\qquad \bigl(\om\in \mathbb{S}^{d-1}\bigr)\,.
$$
In particular, $T_\k$ leaves each subspace $\H_n$ invariant and ${\rm rank}(T_\k|_{\H_n})\leq 1$, 
with equality if and only if the eigenvalue
$$
\la_{d,\k,n} \= \frac{1}{2^dC_n^{d/2-1}(1)}\sum_{\ep\,\in\,\{\pm 1\}^d} \ep_1\cdots\ep_\k\,
 C_n^{d/2-1}\biggl(\frac{\ep_1+\cdots+\ep_d}{{d}}\biggr) 
$$ 
is nonzero, in which case $P_{d,\k,n}$ is the unique associated eigenvector of $T_\k$ in $\H_n$.
\end{lemma}

The next lemma gives part of the structure of the eigenvalues $\la_{d,\k,n}$.

\begin{lemma}
We have $\la_{d,\k,n}=0$ if $0\leq n<\k$ or if $n-\k$ is odd. Moreover, 
$$\la_{d,\k,\k} \= \frac{2^\k\k!(d/2-1)_\k}{d^\k(d-2)_\k}.$$
\end{lemma}

\begin{proof}
Since $\sum_{\ep \in \{\pm 1\}^d} \ep_1\cdots\ep_\k$ vanishes for $1\le\k\le d$, it follows that
\be\label{eq:monimialsum}
\frac 1{2^d}\sum_{\ep\,\in\,\{\pm 1\}^d} \ep_1\cdots\ep_\k\,(\ep_1+\cdots+\ep_d)^m
 \=\frac1{2^d}\sum_{\substack{m_1+\cdots+m_d=m \\ m_1,\dots,m_d\geq 0 \\ 
      m_1,\dots,m_\k \text{ odd} \\ m_{\k+1},\dots,m_d \text{ even}}}\frac{m!}{m_1!\cdots m_d!},
\ee
which implies that \eqref{eq:monimialsum} vanishes if $0\leq m<\k$ or if $m-\k$ is odd. Since $C^{d/2-1}_n(z)$ is an
 even or odd polynomial depending whether $n$ is even or odd, we conclude that $\la_{d,\k,n}=0$ if $0\leq n<\k$ or 
if $n-\k$ is odd. Finally, if $m=n=\k$ then \eqref{eq:monimialsum} equals~$\k!$, so we obtain
$$
\la_{d,\k,\k} \= \frac{\k! \times \text{leading coefficient of } C_\k^{d/2-1}(z) }{C_\k^{d/2-1}(1)} 
\= \frac{2^\k\k!\,(d/2-1)_\k}{d^\k\,(d-2)_\k}\,.
$$
\end{proof}

We conjecture that $\la_{d,\k,\k}=\max \{\la_{d,\k,n}\}_{n\geq 0}$ for all $d\geq 4$ and $\k=0,1,2,3$, but we
could not prove that. Note that we have completely diagonalized the quadratic form on \eqref{eq:projrepres}; however, $\la_{d,\k,\k}$ is only bounded for $d= 3$, which is exactly the case when the weight $w_d(x)\equiv 1$ in Proposition \ref{prop:ufPEidentity}. We now restrict our attention to the $3$-dimensional case.

\begin{lemma}\label{lem:genfunc}
For $0\le\k\le3$ the series $G_\k(w)=\sum_{n\geq 0}\la_{3,\k,n}w^n$ is given by~\eqref{G0G1G2G3}. 
\end{lemma}

\begin{proof}
Using the generating function \eqref{eq:gengegen} of the Gegenbauer polynomials\footnote{Note that if $d=3$ then $C_n^{1/2}(z)$ is 
the Legendre polynomial and $C_n^{1/2}(1)=1$.} for $d=3$ we obtain
\begin{align*}
G_\k(w) & \=\frac{1}{8}\sum_{\varepsilon\in\{\pm 1\}^3}\ep_1\cdots\ep_\k \sum_{n\geq 0} 
      w^nC_n^{3/2-1}\left(\frac{\om\cdot \ep}{\sqrt{3}}\right)  \\
   &\= \frac{1}{8}\sum_{\varepsilon\in\{\pm 1\}^3}\frac{\ep_1\cdots\ep_\k}{\sqrt{1-\frac23(\varepsilon_1+\varepsilon_2+\varepsilon_3)w+w^2}}\;,
\end{align*}
which is easily seen to be equivalent with~\eqref{G0G1G2G3}. 
\end{proof}

\begin{lemma}\label{lem:boundslambda}
Let $\al_{\k}$ and $\b_\k$ be the two largest eigenvalues in $\Lambda_\k=\bigcup_{n\geq 0}\{\la_{3,\k,n}\}$. Then
\begin{alignat}{7}
 \al_0 & \,=\, \la_{3,0,0}=1 &\quad>\quad& 
       \b_0\,=\,\la_{3,0,10}&\=&\tfrac{8320}{19683}&\=&0.4226997\cdots,\\
 \al_1 & \,=\, \la_{3,1,1}=\la_{3,1,5}=\tfrac13 &\quad>\quad& 
       \b_1\,=\,\la_{3,1,15}&\=&\tfrac{469136}{1594323}&\=&0.2942540\cdots,\\
 \al_2 & \,=\, \la_{3,2,2}=\tfrac13 &\quad>\quad& 
       \b_2\,=\,\la_{3,2,8}&\=&\tfrac{232}{729}&\=&0.3182441\cdots\,,\\
 \al_3 & \,=\, \la_{3,3,3}=\tfrac59 &\quad>\quad& 
       \b_3\,=\,\la_{3,3,13}&=&\tfrac{221312}{531441}&\=&0.4164375\cdots\,.
 \end{alignat}
\end{lemma}
\begin{proof}  
One can check that the coefficient of $w^n$ in $(1-\frac23w+w^2)^{-1/2}$ is bounded in absolute value 
by $(3/\pi n\sqrt 2)^{1/2} = 0.82172\cdots n^{-1/2}$, and from this it follows that $\la_{3,\k,n}<\b_\k$
for all $n>1000$ with~$\b_\k$ as in the lemma.  Then by explicitly computing the first 1000
eigenvalues $\la_{3,\k,n}$ we can find the first and second biggest eigenvalues and their multiplicities exactly.
\end{proof}

\begin{proof}[\bf Proof of Theorem \ref{thm:paritymain}]
Noting that if $f\in L_\k^3(\r^3)$ then $\ft f \in L_\k^2(\r^3)$, we then use \eqref{eq:projrepres} to obtain
$$
Q(f) \= \int_{\r^3} |P_E(\ft f)(x)|^2 \d x \= \int_{\r^3} T_\k(\ft f)(x)\;\overline{\ft f(x)}\,\d x.
$$
We can now apply Lemmas~\ref{lem:projEformula},~\ref{prop:Teigenvectors} and~\ref{lem:genfunc} and Plancherel's theorem
to obtain the desired orthogonal decomposition. To finish, Lemmas~\ref{lem:boundslambda} and~\ref{prop:Teigenvectors} 
tell us what the largest eigenvalues are and how to compute their associated eigenspaces, and direct computation then gives
\begin{align*}
 P_{3,0,0}(x) &\= 1\,, \qquad\quad P_{3,1,1}(x) \= \tfrac{1}{\sqrt{3}}x_1\,, \\
 P_{3,1,5}(x)&\=\tfrac{1}{6\sqrt{3}}x_1\,\bigl(-x_1^4 + (5x_2^2 + 5x_3^2)x_1^2 + (-15x_2^4 + 75x_3^2x_2^2 - 15x_3^4)\bigr)\,, \\
 P_{3,2,2}(x)&\=x_1x_2\,, \qquad P_{3,3,3}(x)\=\tfrac{5}{\sqrt{3}}x_3x_2x_1\,.
\end{align*}
This finishes the proof.
\end{proof}

\section{A more general eigenvalue problem}\label{sec:moregen}
{There is another approach to our problem, though rather more technical than the one presented above},  
by using Laguerre polynomials as the basis of {$L_\k^2(\R^3)$} (as in~\cite{Gon2}) to reduce Theorem~\ref{thm:paritymain} 
to a statement about a specific sequence of matrices of increasing size, with entries defined by certain power series
and with the property that the eigenvalues of the larger matrices contain the eigenvalues of the smaller ones. 
In this section we give a more general matrix problem and its rather simple solution, since this can potentially 
be used for other problems.   The application to our original problem is described at the end.

\subsection*{General problem} To any power series $F(w_1,w_2,w_3)\in\cp[w_1,w_2,w_3]$ we associate an endomorphism 
$\Phi_F$ of the vector space~$V=\C[[u,v]]$, defined on monomials by
\be\label{matrix} \Phi_F\biggl[\binom Sa\,u^a\,v^{S-a}\biggr] 
 \= \sum_{b=0}^S\binom Sb \,R_S(a,b)\,u^b\,v^{S-b} \qquad(0\le a\le S)\,, \ee
where 
\be\label{defRS} R_S(a,b) \= \bigl[w_1^aw_2^{S-a}w_3^b\bigr](F)\qquad(0\le a,\,b\le S)\,. \ee
In this section we want to explicitly compute the eigenvalues of $\Phi_F$. Note that $\Phi_F$ preserves the space $V_S$ of homogeneous polynomials of degree~$S$ for every integer $S\ge0$
and $R_S=\bigl(R_S(a,b)\bigr)_{0\le a,b\le S}$ is the matrix representation of the restriction of $\Phi_F$ 
to $V_S$ with respect to the basis $\bigl\{\binom Sa\,u^a\,v^{S-a}\bigr\}_{0\le a\le S}\,$. We will call $F$ 
{\it special} if the map~$\Phi_F$ commutes with multiplication by~$(u+v)$. If this holds, then by considering the 
action of $\Phi_F$ with respect to the basis $\bigl\{u^i(u+v)^{n-i}\bigr\}_{0\le i\le S}$ of~$V_S$ it is easy to deduce that 
the matrix $R_S$ has the eigenvalues $(e_0,e_1,\dots,e_S)$, where $e_n$ is independent of~$S$ and is given by
\be\label{ev}
e_n \= \Phi_F[u^n](1,-1) \= \sum_{b=0}^n(-1)^{n-b}\binom nbR_n(n,b)
 \= \bigl[w_1^nw_3^n\bigr]\bigl((1-w_3)^nF(w_1,0,w_3)\bigr)\,.\ee

We remark that any endomorphism~$\Phi$ of $\C[[u,v]]$ that commutes with the Euler operator 
$u\frac\partial{\partial u}+v\frac\partial{\partial v}$ preserves each $V_S$ and hence gives rise 
via~\eqref{matrix} to a sequence of $(S+1)\times(S+1)$ matrices~$R_S$ and via~\eqref{defRS}
to a power series $F(w_1,w_2,w_3) =\sum_{a,b,c\ge0}R_{a+c}(a,b)w_1^aw_2^cw_3^b$ with $\Phi_F=\Phi$. 
(For example, if $\Phi$ is the identity then $F$ is $(1-w_2)^{-1}(1-w_1w_3)^{-1}$.) 
Equation~\eqref{ev} gives a very simple formula for the eigenvalues of the matrix~$R_S$ if 
the power series $F$ is special.  The next proposition tells us how to recognize when $F$ is special and gives an alternative way to compute the eigenvalues of $\Phi_F$ in \eqref{ev}.

\begin{proposition}\label{general} Let $F\in\C[[w_1,w_2,w_3]]$.  Then the following are equivalent.
\newline\noindent\phantom{XX}(i)  $F$ is special, i.e., $\Phi_F$ commutes with multiplication by~$(u+v)\,$.
\newline\noindent\phantom{XX}(ii) $F$ satisfies the first order linear differential equation
\begin{align}\label{LDE} 
(1-w_1)\,\frac{\partial F}{\partial w_1} \+ (1-w_2)\,\frac{\partial F}{\partial w_2} \+ 
w_3(1-w_3)\,\frac{\partial F}{\partial w_3} \= (1+w_3)\,F\,.
\end{align}
\newline\noindent\phantom{XX}(iii) $F$ has the form
\be \label{FfromH} F(w_1,w_2,w_3)\= \frac1{(1-w_2)(1-w_3)}\,H\Bigl(\frac{w_1-w_2}{1-w_2},\,w_3\,\frac{1-w_2}{1-w_3}\Bigr) \ee
\newline\noindent\phantom{XX}for some power series $H(x,y)\in\cp[[x,y]]\,$.
\newline\noindent If these conditions hold, then the eigenvalues of $\Phi_F$ are given by
 \be \label{eigen}  e_n \ = \bigl[x^ny^n\bigr](H) \qquad(n\ge0)\,. \ee
\end{proposition}
\noindent{\it Example}: The $H$ corresponding to $\Phi=$Id is $(1+y-xy)^{-1}$, with $e_n=1$ for all~$n$.
\begin{proof} 
For $a,b,c\ge0$ set $R(a,b,c)=[w_1^aw_2^bw_3^c](F)=R_{a+b}(a,c)$. Then both conditions~(i) and~(ii)
translate to the {\it same} recursion
\be (a+1)R(a+1,b,c) + (b+1)R(a,b+1,c) \,=\, cR(a,b,c-1) + (a+b-c+1)R(a,b,c), \ee
as one checks by an easy direct calculation.  This proves the equivalence of~(i) and~(ii).  For that
of (ii) and~(iii) we use the inverse isomorphisms $\cp[[x,y,z]]\cong\cp[[w_1,w_2,w_3]]$ given by 
\begin{align}      H(x,y,z)\;\mapsto\; F(w_1,w_2,w_3)
     &\= \frac1{(1-w_2)(1-w_3)}\,H\Bigl(\frac{w_1-w_2}{1-w_2},\,w_3\,\frac{1-w_2}{1-w_3},\,w_2\Bigr) \,, \\
 F(w_1,w_2,w_3)\;\mapsto\; H(x,y,z) &\= \frac{(1-z)^2}{1+y-z}\,F\bigl(x+z-xz,\,z,\,\frac y{1+y-z}\bigr)\,. 
\end{align}
Under this correspondence the differential equation \eqref{LDE} is sent to $\dfrac{\partial H}{\partial z}=0\,$.
Finally, if~\eqref{FfromH} holds with $H(x,y)=\sum\limits_{i,j\ge0}h_{ij}x^iy^j\,$, then
\begin{align*}
\Phi\bigl[u^n\bigr] &\= \sum_{b=0}^n\binom nb\,\bigl[w_1^nw_3^b\bigr]
\Biggl(\sum_{i,j\ge0} h_{ij}\,\frac{w_1^i\,w_3^j}{(1-w_3)^{j+1}}\Biggr)\,u^bv^{n-b} \\
 &\= \sum_{0\le j\le b\le n} \binom nb\,\binom bj\, h_{nj}\,u^bv^{n-b} 
 \= \sum_{j=0}^n\binom nj\,h_{nj}\,u^j(u+v)^{n-j}\,, 
 \end{align*}
and therefore equation~\eqref{ev} gives $e_n=h_{nn}$ as claimed.  Note that this proof shows that the matrix $R_S$ is
conjugate to the triangular matrix~$R_S^*$ with entries $R_S^*(a,b)=\binom ab\,h_{ab}\,$.
\end{proof}

In the remainder of this section we will explain how Proposition~\ref{general} can be applied to the problem described
in the introduction and whose solution was given in~Theorem~\ref{thm:paritymain}. {For $n\in\Z_{\ge0}$ and $\ep\in\{\pm1\}$ we define $Q_n^{\ep}\in L^2(\R)$ by $ Q^{\ep}_n(x)\=x^r L_n^{(-\ep/2)}(\pi x^2)\,e^{-\pi x^2/2}$, where $\ep = (-1)^r$ with $r \in \{0,1\}$ and $L_n^{(-\ep/2)}(z)$ is 
the Laguerre polynomial of degree~$n$ and parameter~$-\ep/2$. These functions can be defined by the generating series identity}
$$ \QQ^{\ep}(x;t)\;=\; \sum_{n=0}^\infty  Q^{\ep}_n(x)\,t^n 
     \= \frac{x^r}{(1-t)^{r+1/2}}\,\exp\Bigl(-\frac\pi2\,\frac{1+t}{1-t}\,x^2\Bigr)\;. $$ 
The easy calculation
  $\int_{-\infty}^\infty \QQ^{\ep}(x;t_1)\QQ^{\ep}(x;t_2)\d x\,=\,\frac{\pi^{-1/2}\,\G(r+1/2)}{(1-t_1t_2)^{r+1/2}}$
implies that the $Q_n^{\ep}$ are orthogonal with $\|Q_n^{\ep}\|_{L^2(\r)}^2=\frac{\G(n+r+1/2)}{\G(1/2)\,n!}\,$.
In particular, these functions are orthogonal bases for $L^2_\ep(\R)$, and for each integer $0\le\k\le 3$ the multivariate 
functions
$$ Q_\bn(x)\= Q^{(\k)}_\bn(x)\=\prod_{1\le j\le\k} Q^-_{n_j}(x_j)\,\prod_{\k<j\le 3}Q^+_{n_j}(x_j)
   \qquad(\bn\in\Z_{\ge0})$$
form an orthogonal basis of the parity~$\k$ space $L^2_\k(\R^3)$ as defined in~\eqref{L2k}.  Instead of appealing 
to Proposition~\ref{prop:ufPEidentity}, we can then deal directly with the quadratic form \eqref{def:Q} by applying Lemma~11 of~\cite{Gon} to this orthonormal basis of~$L^2_\k(\R^3)$.
This leads after some computations to the problem of computing the eigenvalues of the sequence of 
$\binom{S+2}{2}\times\binom{S+2}{2}$ matrices $M_S=\bigl((\wt Q_\bm,\wt Q_\bn)_{L^2(\R)}\bigr)_{|\bm|=|\bn|=S}$, 
where $\wt Q\in L^2(\R)$ is defined by $\wt Q(x)=Q(x\,\1)$, with $\1=\tfrac1{\sqrt{3}}(1,1,1)$. 
In the rest of this section we show how to solve this problem.

The first point is that $M_S$ has the same non-zero eigenvalues (with multiplicities) as a much smaller matrix. To see this, we observe that $M_S$ factors as $A_S^*A_S$, where $A_S$ is the map defined 
by $Q\mapsto\wt Q$ from the  $\binom{S+2}{2}$-dimensional subspace $V_S=\langle Q_\bn\rangle_{|\bn|=S}$ 
of $L^2(\R^3)$ {to the $(S+1)$-dimensional space of $L^2(\R)$
$$
\wt V_{S}=\bigl\langle x^{\k+2n} e^{-\pi x^2/2}\bigr\rangle_{0\le n\le S}.
$$
}Hence $M_S$ has the same non-zero eigenvalues as the endomorphism $\wt M_S=A_SA_S^*$ of $\wt V_S$.  To compute 
them we will actually consider $\wt M_S$ acting in the larger subspace
$$
\wt{ \wt {V}}_{S} =\bigl \langle Q^\ep_n\bigr\rangle_{0\le n\le S+\k/2},
$$
where $\ep=(-1)^\k$ (i.e., we consider $A_S:V_S \to \wt{ \wt {V}}_{S}$ and take its adjoint $A_S^*$ in the 
larger subspace $\wt{ \wt {V}}_{S}$). Then we will show that the matrix representation of $\wt M_S$ is of the 
form~\eqref{defRS} for some special power series $F=F_{\k}$, but with $S$ replaced by $S+[\k/2]$.

To do this we use kernel functions.  The two 3-variable generating series 
\be\label{defPpm} \P^\ep(x,y;w) \;=\; \frac{e^{-\pi\frac{1+w}{1-w}\,\frac{x^2+y^2}2}}{\sqrt{1-w}} \;\times\; 
 \begin{cases} \cosh\bigl(\frac{2\pi xy\sqrt w}{1-w}\bigr) \ \ \text{ if } \ep=+1, \\ \frac1{\sqrt w}\,\sinh\bigl(\frac{2\pi xy\sqrt w}{1-w}\bigr)\ \ \text{ if } \ep=-1\end{cases}\ee
are the kernel functions for the multiplication maps on $L^2_\ep(\R)$ sending $Q_n^\ep$ to $w^nQ_n^\ep$ for all~$n\ge0$.
We can see this either {by} using formula~8.976-1 of~\cite{GR} to write
$$ \P^\ep(x,y;w) \= \sum_{n=0}^\infty \frac{Q_n^\ep(x)\,Q_n^\ep(y)}{\|Q_n^\ep\|^2}\,w^n\,,$$
or else directly by using the algebraic identity
$$  \frac{1+w}{1-w}\,\frac{x^2+y^2}2\,-\,\frac{2xy\sqrt w}{1-w}\+\frac{1+t}{1-t}\,\frac{y^2}2 
 \= \frac{1+wt}{1-wt}\,\frac{x^2}2 \+ \frac{1-wt}{(1-w)(1-t)}\,\Bigl(y-\frac{x(1-t)\sqrt w}{1-wt}\Bigr)^2 $$
to get (setting $p=[\k/2], \, r=\k-2p, \, \ep = (-1)^r \= (-1)^\k$)
\begin{align} &\int_{-\infty}^\infty \P^\ep(x,y;w)\,\QQ^\ep(y;t)\d y 
 \= \frac{w^{-r/2}\,e^{-\pi\frac{1+wt}{1-wt}\,\frac{x^2}2}}{\sqrt{1-w}\;(1-t)^{r+1/2}}\,
    \int_{-\infty}^\infty y^r\,e^{-\pi\frac{1-wt}{(1-w)(1-t)}\,\bigl(y-\frac{x(1-t)\sqrt w}{1-wt\vphantom{a_{y_y}}}\bigr)^2}\d y \\
  & \qquad\qquad \= \frac{w^{-r/2}\,e^{-\pi\frac{1+wt}{1-wt}\,\frac{x^2}2}}{\sqrt{1-w}\,(1-t)^{r+1/2}}\,
  \sqrt{\frac{(1-w)(1-t)}{1-wt}}\;\Bigl(\frac{x(1-t)\sqrt w}{1-wt}\Bigr)^r \= \QQ^\ep(y;wt)\,. \label{integral} \end{align}
(We have given {this second proof} because a more complicated form of the same calculation will be used in a moment.) It follows that the product 
$$\P^{(\k)}(x,y;w) =\prod_{1\le j\le \k}\P^-(x_j,y_j;w)\,\prod_{\k<j\le 3}\P^+(x_j,y_j;w)$$
is the multivariate kernel function for the multiplication map $Q_\bn^{(\k)}\mapsto w^{|\bn|}Q_\bn^{(\k)}$ on~$L^2_\k(\R^3)$ and
hence that the coefficient of~$w^S$ in its diagonal restriction
$$ \wt\P^{(\k)}(x,y;w) \= \P^{(\k)}(x\,\1,y\,\1;w) \= 
\P^+\Bigl(\frac x{\sqrt 3},\frac y{\sqrt 3};w\Bigr)^{3-\k}\,\P^-\Bigl(\frac x{\sqrt 3},\frac y{\sqrt 3};w\Bigr)^\k $$
is the kernel function of the endomorphism~$\wt M_S$ of~${\wt {\wt V}}_S$. 
From~\eqref{defPpm} we get the explicit formula{
$$ \wt\P^{(\k)}(x,y;w) \= \frac{e^{-\pi\frac{1+w}{1-w}\,\frac{x^2+y^2}2}}{w^{\k/2}\,(1-w)^{3/2}} \; 
 \sum_{m\in\{\pm3,\pm1\}} C^{(\k)}_m\,e^{\frac{2\pi mxy\sqrt w}{3(1-w)}}\,, $$
}where {$C_m^{(\k)}$ are defined by $\cosh^{3-\k}\!x\,\sinh^\k\!x=\sum C_m^{(\k)}\,e^{mx}$, or explicitly by $C_{\pm3}^{(\k)} = (\pm1)^\k /8$ and $C_{\pm1}^{(\k)} =(\pm1)^\k (3-2\k)/8$.}

To compute the action of~$\wt M_S$ on ~${\wt {\wt V}}_S$ we must compute the scalar product of~$\wt\P^{(\k)}(x,\,\cdot\,)$ and~$Q_n^{\ep}$,
which we again do by using generating functions and a Gaussian integral calculation. The calculation is exactly similar to the one given in~\eqref{integral}: using the algebraic identity
$$  \frac{1+w}{1-w}\,\frac{x^2+y^2}2\,-\,\frac{2mxy\sqrt w}{3(1-w)}\+\frac{1+t}{1-t}\,\frac{y^2}2 
 \= \frac{1+z_m}{1-z_m}\,\frac{x^2}2 \+ \frac{1-wt}{(1-w)(1-t)}\,\Bigl(y\,-\,\frac{mx(1-t)\sqrt w}{3(1-wt)}\Bigr)^2\,, $$
where 
  \be z_m \= w\,\frac{1-wt-m^2(1-t)/9}{1-wt-m^2w(1-t)/9} \= \begin{cases} \ \ \ \ \ \ \ wt   & \text{ if } m=\pm 3, \\ 1-\frac{(1-w)(1-wt)}{1-\frac89 wt-\frac19w}  & \text{ if } m\=\pm1,\end{cases} \ee
we find 
 $$\int_{-\infty}^\infty \P^{(\k)}(x,y;w)\,\QQ^{\ep}(y;t)\d y \={ \sum_{m\in\{\pm3,\pm1\}}} C^{(\k)}_m\, I^{(\k)}_m(x;w,t) $$
with  
\begin{align} 
 &I^{(\k)}_m(x;w,t) \= \frac{w^{-\k/2}\,e^{-\pi\frac{1+z_m}{1-z_m}\,\frac{x^2}2}}{(1-w)\,(1-t)^{r+1/2}}\,
     \int_{-\infty}^\infty y^r\,e^{-\pi\frac{1-wt}{(1-w)(1-t)}\bigl(y-\frac{mx(1-t)\sqrt w}{d(1-wt)}\bigr)^2}\d u \\
  & \qquad\quad \= \frac{w^{-p}\,e^{-\pi\frac{1+z_m}{1-z_m}\,\frac{x^2}2}}{1-w}\,\frac{(mx/3)^r}{(1-wt)^{r+1/2}}
   \= \frac{w^{-p}\,(m/3)^r}{1-w}\, \Bigl(\frac{1-z_m}{1-wt}\Bigr)^{r+1/2}\, \QQ^\ep(x;z_m) \,.  \end{align}
Now replacing each $Q_b^\ep(y)$ by~$u^b$ to keep track of the matrix coefficients, inserting the values of $C^{(\k)}_m$ and noticing that the resulting series is a power series in $w,u$ and $T=wt$, we find 
\be 
\wt M_{S}(a,b)  \=[T^a w^{S+p-a}u^b](F_\k )\; \ee
for $0\leq a,b\leq S+p$, where
\be
F_\k(T,w,u)\= \frac14\Bigl(\frac{1-z_3}{1-T}\Bigr)^{r+1/2}\;\frac{1}{1-z_3u} \,+\,  \frac{3-2\k}4 \Bigl(\frac{1-z_1}{1-T}\Bigr)^{r+1/2}\;\frac{(1/3)^r}{1-z_1u}.
\ee
But this function is special because it has the form \eqref{FfromH} with $(w_1,w_2,w_3)=(T,w,u)$ and
$$
H_\k (x,y) \= \frac14\,\frac1{ 1-y(x-1)} \+ \frac{3-2\k}{4\cdot3^r}\, \frac{(1-\tfrac89 x)^{1/2-r}}{1-\tfrac89 x-y(x-1)}\,.
$$
Therefore Proposition \ref{general} gives the eigenvalues of $\wt M_S$ as
\begin{align*}
e_n^{(\k)} & \=[x^ny^n](H_\k) \=[x^n]\Bigl(\frac{(x-1)^n}{4}\Bigl(1\+\frac{3-2\k}{3^r(1-8x/9)^{n+r+1/2}}\Bigr)\Bigr)  \\ 
& \= \frac14 \+ \frac14 \, \frac{3-2\k}{3^r} \, [v^n] 
{\Bigl(\tfrac1{\sqrt{1+14v/9+v^2}}\,\Bigl(\tfrac{1+v+\sqrt{1+14v/9+v^2}}2\Bigr)^{1/2-r} \Bigr)\,, }
\end{align*}
where in the second line we have used the substitution $v=x \, \frac{1-8x/9}{x-1}$ and the residue theorem, and it 
is easily checked that this equals the coefficient of $w^{2n+r}$ in the series $G_\k(w)$ defined in equation~\eqref{G0G1G2G3}.
It is also possible with this method to characterize the eigenspaces associated with the largest eigenvalues, but 
this requires a bit more work (especially in the case $\k=1$ where the space $V_\k(\a_\k) = V_1(\tfrac13)$ has rank ~2) 
and is not carried out here.

\end{document}